\documentclass[11pt,a4paper]{article}

\usepackage{epsf,epsfig,amsfonts,amsgen,amsmath,amstext,amsbsy,amsopn,amsthm
}
\usepackage{amsmath,times,mathptmx}
\usepackage{enumitem}
\usepackage{amsfonts,amsthm,amssymb}
\usepackage{amsfonts}
\usepackage{graphics}
\usepackage{latexsym,bm}
\usepackage{amsfonts,amsthm,amssymb,bbding}
\usepackage{indentfirst}
\usepackage{graphicx}
\usepackage{color}
\usepackage[colorlinks=true,anchorcolor=blue,filecolor=blue,linkcolor=blue,urlcolor=blue,citecolor=blue]{hyperref}
\usepackage{float}
\usepackage{tikz}

\DeclareMathAlphabet{\mathcal}{OMS}{cmsy}{m}{n}
\DeclareSymbolFont{largesymbols}{OMX}{cmex}{m}{n}
\allowdisplaybreaks[4]
\usepackage{moresize}

\newtheorem{thm}{Theorem}

\newtheorem{lemma}{Lemma}
\newtheorem{false statement}{False statement}

\theoremstyle{definition}
\newtheorem{definition}{Definition}
\newtheorem{claim}{Claim}

{Remark}

\newtheorem{casee}{Case}
\newtheorem{caseee}{Case}

\newtheorem{subcasee}{Subcase}[casee]
\newtheorem{subcaseee}{Subcase}[caseee]

\begin{document}
	
\title{Spectral extremal problems on planar and outerplanar graphs without $C_{k,l}$
\footnote{Supported by Tianshan Talent Training Program (No. 2024TSYCCX0013),
	Natural Science Foundation of Xinjiang Uygur Autonomous Region (No. 2024D01C41),
	 The Basic scientific research in universities of Xinjiang
	Uygur Autonomous Region (XJEDU2025P001) and NSFC (No. 12361071).}}
\author{{Jiamin Li, Dan Li\thanks{Corresponding author. E-mail: ldxjedu@163.com.}, Xilong Yin, Yuanyuan Chen}\\
{\footnotesize College of Mathematics and Systems Science, Xinjiang University, Urumqi 830046, China}}
\date{}
	
\maketitle {\flushleft\large\bf Abstract:}
Let $\emph{spex}_{\mathcal{P}}(n,F)$ and $\emph{spex}_{\mathcal{OP}}(n,F)$ be the maximum spectral radius among all $n$-vertex $F$-free planar graphs and outerplanar graphs, respectively. Define $C_{k,l}$  as a graph obtained from $C_k \cup C_l$ such that the two cycles share a common vertex, where  $l \ge k \ge 3$. 
In the 1990s, Cvetkovi\'c and Rowlinson conjectured $K_1 + P_{n-1}$ maximizes spectral radius in outerplanar graphs on $n$ vertices, while Boots and Royle (independently, Cao and Vince) conjectured $K_2 + P_{n-2} $ does so in planar graphs.
Tait and Tobin [J. Combin. Theory Ser. B, 2017] determined the fundamental structure as the key to confirming these two conjectures for sufficiently large $n$.  Recently, Yin and Li [Discrete Mathematics, 2026] characterized the extremal graphs for $\emph{spex}_{\mathcal{P}}(n,B_{t,l})$ and $\emph{spex}_{\mathcal{OP}}(n,B_{t,l})$ in planar and outerplanar graphs on the basis of this key idea, where $B_{t,l}$ denotes the graph obtained by  $t$ edge-disjoint $l$-cycles sharing a common vertex. 
In this paper, we focus on planar and outerplanar graphs without $C_{k,l}$, and  determine $\emph{spex}_{\mathcal{P}}(n,C_{k,l})$ and $\emph{spex}_{\mathcal{OP}}(n,C_{k,l})$  along with their unique extremal graphs for all $l \geq  k \geq 3$ and large $n$.

\vspace{0.1cm}
\begin{flushleft}
\end{flushleft}

\section{Introduction}
	 %A graph is planar if it can be embedded in the plane, that is, it can be drawn on the plane in such a way that edges intersect only at their endpoints. A graph is outerplanar if it can be embedded in the plane so that all vertices lie on the boundary of its outer face. Let $\mathcal{F}$ be a given family of graphs. If a graph $G$ does not contain any member of $\mathcal{F}$ as a subgraph, then $G$ is $\mathcal{F}$-free. If $\mathcal{F}={F}$, we write $G$ is $F$-free. Let $G = (V(G), E(G))$ be a graph, where $V(G)$ and $E(G)$ are the vertex set and the edge set of $G$. For two given graphs $G_1$ and $G_2$, the joint product of $G_1$ and $G_2$, denoted by $G_1+G_2$, is obtained from $G_1$ and $G_2$ by adding an edge between an arbitrary vertex in $G_1$ and an arbitrary vertex in $G_2$. The degree of $v$, denoted by $d_G(v)$, is the number of the edges which are incident with $v$. Let $A(G)$ be the adjacency matrix of $G$. The spectral radius of $G$, denoted by $\rho(G)$, is the largest modulus of all the eigenvalues of $A(G)$. Let $\boldsymbol{x} = (x_1, x_2, \ldots, x_n)^{\mathrm{T}}$ be an eigenvector of $A(G)$ corresponding to $\rho(G)$. 
	 
Let $G = (V(G), E(G))$ be a graph, where $V(G)$ and $E(G)$ are the vertex set and the edge set of $G$, respectively. A graph is planar if it admits a plane embedding with no edge crossings except at vertices, and outerplanar if such an embedding exists with all vertices on the outer face boundary. Let $\mathcal{F}$ be a family of graphs. A graph $G$ is said to be $\mathcal{F}$-free if it does not contain any member of $\mathcal{F}$ as a subgraph. If $\mathcal{F} = \{F\}$, we simply write that $G$ is $F$-free. Let $A(G)$ be the adjacency matrix of $G$. The spectral radius of $G$, denoted by $\rho(G)$, is the largest modulus of the eigenvalues of $A(G)$. Let $\emph{spex}_{\mathcal{P}}(n,F)$ and $\emph{spex}_{\mathcal{OP}}(n,F)$ be the maximum spectral radius among all $n$-vertex $F$-free planar graphs and outerplanar graphs, respectively. Denote the family of  $n$-vertex  $F$-free planar and outerplanar with maximum spectral radius by $\operatorname{SPEX}_{\mathcal{P}}(n,F)$ and $\operatorname{SPEX}_{\mathcal{OP}}(n,F)$, respectively. For a vertex $v \in V(G)$, $N_G(v)$ denotes the set of neighbors of $v$ in $G$, and $d_G(v)=|N_G(v)|$. For two given graphs $G_1$ and $G_2$, denote the $join$ of $G_1$ and $G_2$ by $G_1 + G_2$, where $join$ means the graph obtained from $G_1$ and $G_2$ by adding all possible edges between $V(G_1)$ and $V(G_2)$.  
%Let $\boldsymbol{x} = (x_1, x_2, \ldots, x_n)^{\mathrm{T}}$ be an eigenvector of $A(G)$ corresponding to $\rho(G)$. The Rayleigh quotient of $G$ is expressed as
%	$$	\rho(G) = \max_{\boldsymbol{x} \in \mathbb{R}^n} \frac{\boldsymbol{x}^{\mathrm{T}} A(G) \boldsymbol{x}}{\boldsymbol{x}^{\mathrm{T}} \boldsymbol{x}}
%		= \max_{\boldsymbol{x} \in \mathbb{R}^n} \frac{2 \sum_{uv \in E(G)} x_u x_v}{\boldsymbol{x}^{\mathrm{T}} \boldsymbol{x}},$$
%	where $\mathbb{R}^n$ is the set of $n$-dimensional real numbers.
%If $G$ is a connected graph with $n$ vertices, then by the well-known Perron-Frobenius theorem, there exists a unique positive vector $\boldsymbol{x} = \{x_1, \ldots, x_n\}^{\mathrm{T}}$ corresponding to $\rho(G)$. It is convenient for us to normalize $\boldsymbol{x}$ such that the maximum entry of $\boldsymbol{x}$ is 1. Usually, $\boldsymbol{x}$ is referred to as the Perron vector of $G$.

The \emph{Brualdi-Solheid problem} \cite{R.A. Brualdi}, a classic problem in spectral graph theory, seeks the graph with the maximum spectral radius within a given class. 
In 2010, Nikiforov \cite{V. Nikiforov-4} proposed a variation of Brualdi-Soheid problem which asks: given a graph $F$, what is the maximum spectral radius of an $F$-free graph of order $n$?
This problem has been extensively studied for various forbidden subgraphs, such as complete graphs \cite{B. Bollobás}, paths \cite{V. Nikiforov-4}, cycles \cite{L.F. Fang-2, V. Nikiforov-4, M.Q. Zhai-1},
matchings \cite{L.H. Feng}, wheels \cite{S. Cioab˘a-2,Y.H. Zhao}, star forests \cite{M.Z. Chen-1} and linear forests \cite{M.Z. Chen-2}. For more about this topic, we refer readers to surveys \cite{D. Cvetković,  Y.T. Li, V. Nikiforov-3}.

Spectral extremal problems on planar and outerplanar graphs have attracted wide attention. In 1978 Schwenk and Wilson \cite{A. J. Schwenk} asked the fundamental question: What can be said about the eigenvalues of a planar graph?
Subsequently, Boots and Royle \cite{B. N. Boots}, and independently Cao and Vince \cite{D. Cao},
conjectured that $K_2 + P_{n-2}$ is the unique planar graph on $n\ge 9$ vertices with maximum spectral radius.
For outerplanar graphs, Cvetkovi\'{c} and Rowlinson \cite{D. Cvetković} conjectured that
$K_1 + P_{n-1}$ maximizes the spectral radius.
These two landmark conjectures were confirmed for sufficiently large $n$ by Tait and Tobin \cite{M. Tait} in 2017,
and the Cvetkovi\'{c}--Rowlinson conjecture was later completely solved by Lin and Ning \cite{H.Q. Lin}.

Define $C_{k,l}$  be a graph obtained from $C_k \cup C_l$ such that the two cycles share a common vertex, where  $l \ge k \ge 3$. A theta graph $\theta_{a,b,c}$ is the union of three internally disjoint paths sharing the same pair of distinct endpoints, where $a,b,c$ denote the lengths of these paths.
%A Theta graph is obtained from a cycle $C_k$ by adding an additional edge between two non-consecutive vertices on $C_k$ , where $k \ge 4$.
In recent years, researchers have further extended the spectral extremal theory to planar and outerplanar graphs with forbidden subgraphs.
Fang, Lin and Shi \cite{L.F. Fang-2} determined $\emph{spex}_{\mathcal{P}}(n,tC_l)$ and characterize the unique extremal graph for $1 \le t \le 2, l \ge 3$ and sufficiently large $n$, where $tC_l$ denotes the disjoint union of $t$ copies of $l$-cycles.
Zhai and Liu \cite{M.Q. Zhai-3} characterized the maximum spectral radius and its unique extremal graph over all $n$ vertex planar graphs without $k$ edge-disjoint cycles. 
Zhang and Wang \cite{H.R. Zhang} studied the spectral extremal problems for $C_{l,l}$-free and Theta-free planar graphs.
Very recently, Yin and Li \cite{Yin2025,yin2026spectral} systematically studied spectral extremal problems for outerplanar and planar graphs
without paths, linear forests and cycles.

Previous research has predominantly focused on extremal problems forbidding cycles of equal length, such as single cycles or unions of cycles of the same order. In this paper, motivated by the above developments, we start focusing on two cycles of different lengths and completely determine the spectral extremal graphs for $C_{k,l}$-free planar and outerplanar graphs for all $l \ge k \ge 3$ and sufficiently large $n$.
%\vspace*{2mm}

\begin{definition}\cite{yin2026spectral}.\label{de1}
For positive integers $n, n_1, n_2$ with $n\geq n_1>n_2\geq 1$. Define $H_{\mathcal{P}}(n_1,n_2)$ as follows:
	\[
	H_{\mathcal{P}}(n_1,n_2)=
	\begin{cases}
		P_{n_1}\cup \frac{n-2-n_1}{n_2}P_{n_2}, & \text{if } n_2\mid (n-2-n_1);\\[6pt]
		P_{n_1}\cup \left\lfloor \frac{n-2-n_1}{n_2} \right\rfloor P_{n_2}\cup P_{n-2-n_1-\left\lfloor \frac{n-2-n_1}{n_2} \right\rfloor n_2}, & \text{otherwise.}
	\end{cases}
	\]
\end{definition}

\begin{definition}\cite{yin2026spectral}.\label{de2}
	For positive integers $n,n_1,n_2$ with $n\geq n_1>n_2 \geq 1$. Define $H_{\mathcal{OP}}(n_1, n_2)$ as follows:
	\[  
	H_{\mathcal{OP}}(n_1, n_2) = \begin{cases}   
		P_{n_1} \cup \frac{n-1-n_1}{n_2} P_{n_2}, & \text{if } n_2|(n-1-n_1); \\  
		P_{n_1} \cup \left\lfloor \frac{n-1-n_1}{n_2} \right\rfloor P_{n_2} \cup P_{n-1-n_1-\left\lfloor \frac{n-1-n_1}{n_2} \right\rfloor n_2}, & \text{otherwise.}  
	\end{cases}  
	\]
	
\end{definition}

Zhang and Wang\cite{H.R. Zhang} characterized $\operatorname{SPEX}_{\mathcal{P}}(n, C_{l,l}) = K_2 + H_{\mathcal{P}}(l-2,l-2)$. Furthermore, Yin and Li \cite{yin2026spectral} determined  $\operatorname{SPEX}_{\mathcal{P}}(n,$ $ B_{t,l})$ and $\operatorname{SPEX}_{\mathcal{OP}}(n,$ $ B_{t,l})$, where $B_{t,l} \cong C_{l,l}$ for $t=2$. Therefore, we consider the cycles of different lengths and restrict our attention to the case $l > k \ge 3$ throughout the paper. The main results of this paper are as follows.

\begin{thm}\label{thm1}
Let $l,k,n$ be positive integers with $l > k \ge 3$, and $n >\max\{ \frac{21k^3 - 159k^2 + 388k - 300}{4},5k^3-38k^2+92k-69, \frac{21l^3 - 255l^2 + 979l - 1145}{32}, \frac{361{\left\lfloor \frac{l-3}{2} \right\rfloor}^2-722{\left\lfloor \frac{l-3}{2} \right\rfloor}-425}{32} ,\frac{10}{9}|V(F)|, 2.67\times9^{17}, 10.2\times2^{{\left\lfloor \frac{l-3}{2} \right\rfloor}}+2\}$. Then $\operatorname{SPEX}_{\mathcal{P}}(n, C_{k,l}) = G$, where
\[G :\cong
\begin{cases}
K_2+H_{\mathcal{P}}(k-2,k-2), & \text{if } l \leq 2k-2; \\[4pt]
K_2+H_{\mathcal{P}}(l-2,k-2), & \text{if } l \in \{2k-1,2k\}; \\[4pt]
K_2+H_{\mathcal{P}}(\left\lceil \frac{l-3}{2} \right\rceil,\left\lfloor \frac{l-3}{2} \right\rfloor), & \text{if } l \geq 2k+1. \\[4pt]
\end{cases}\]

\end{thm}
\vspace*{2mm}

\begin{thm}\label{thm2}
	Let $l,k,n$ be positive integers with $l > k \ge 3$ and $n > \max \{\frac{(k-2)(2l^2 - 14l + 3kl - 6k + 19)}{l - k},$ $400, \frac{5}{4} |V(F)|,5 \times 2^{l} + 1\} $. Then $\operatorname{SPEX}_{\mathcal{OP}}(n, C_{k,l}) = G$, where $G \cong K_1 + H_{\mathcal{OP}}(l-2,l-2).$
\end{thm}

\vspace*{2mm}
\section{Proof of Theorem \ref{thm1}}\label{sec2}

Let $G$ be a connected graph of order $n$, by Perron-Frobenius theorem, then there exists a positive eigenvector $X=(x_1,...,x_n)^T$ corresponding to $\rho(G)$.
Now let $G$ be an extremal graph for $\textit{spex}_\mathcal{P}(n,F)$. For convenience, we now normalize $X$ such that its maximum entry equals $1$. Before proceeding, we first list several useful lemmas.
%\vspace*{2mm}
\begin{lemma}\cite{Wang2025}.\label{lm1}  
Let $F$ be a planar graph not contained in $K_{2,n-2}$, where $n \geq \max\{2.67 \times 9^{17}, \frac{10}{9}|V(F)|\}$.  
Suppose that $G$ is a connected extremal graph in $\operatorname{SPEX}_{\mathcal{P}}(n, F)$ and
$X$
is the positive eigenvector of $\rho := \rho(G)$ with $\max_{v \in V(G)} x_{v} = 1$.   
Then there exist two vertices $u',u''\in V( G)$ such that $R:=N_G(u')\cap N_G( u'')=V( G)\setminus\{u', u''\}$ and $x_{u'}= x_{u''}= 1.$ In particular, $G$ contains a copy of $K_{2, n-2}.$

\end{lemma}
%\vspace*{2mm}
\begin{lemma}\cite{Wang2025}.\label{lm2}
	Suppose further that G contains $K_{2, n- 2}$ as a subgraph. Let $u', u''$ be the two vertices of $G$ that have degree $n-2$ in $K _{2, n- 2}$.
	For any vertex $u\in V( G) \setminus \{ u', u''\} $, we have 
	$$\frac 2\rho \leq x_u\leq \frac 2\rho + \frac {4. 496}{\rho ^2}.$$
\end{lemma}
%\vspace*{2mm}
\begin{definition}\cite{Wang2025}.\label{de3}
	Let $s_1$ and $s_2$ be two integers with $s_1\geq s_2\geq1$, and let $H=P_{s_1}\cup P_{s_2}\cup H_0$, where $H_0$ is a disjoint union of paths. We say that $H^*$ is an $(s_1,s_2)$-transformation of $H$ if
	$$H^*:=\begin{cases}P_{s_1+1}\cup P_{s_2-1}\cup H_0&\text{if }s_2\geq2,\\P_{s_1+s_2}\cup H_0&\text{if }s_2=1.\end{cases}$$
	Clearly, $H^*$ is a disjoint union of paths, which implies that $K_2 + H^*$ is planar.
\end{definition}
%\vspace*{2mm}
\begin{lemma}\cite{Wang2025}.\label{lm3}
Let $H$ and $H^*$ be the two graphs as shown in Definition \ref{de3}.
When $n\geq\max\{2.67\times9^{17},10.2\times2^{s_2}+2\}$, we have $\rho(K_{2} + H^*) > \rho(K_{2} + H)$.
\end{lemma}
%\vspace*{2mm}
\begin{lemma}\cite{Q. Li}.\label{lm4}
Let $G$ be a connected graph, and let $G'$ be a proper spanning subgraph of $G$. Then $\rho(G) > \rho(G')$.
\end{lemma}
%\vspace*{2mm}	
\noindent \emph{\textbf{Proof of Theorem \ref{thm1}}.}
 %\begin{proof}[\textbf{Proof of Theorem \ref{thm1}}]
 	Let $G$ be the extremal graph for $\textit{spex}_\mathcal{P}(n,C_{k,l})$. Noting that $C_{k,l}$ is a planar graph not contained in $K_{2,n-2}$, where
 	$n \geq \max\{2.67 \times 9^{17}, \frac{10}{9}|V(F)|\}$. Then $G$ contains a copy of $K_{2,n-2}$ , and there exist two vertices $u',u''\in V( G)$ such that $R:=N_G(u')\cap N_G( u'')=V( G)\setminus\{u', u''\}$ and $x_{u'}= x_{u''}= 1$ by Lemma \ref{lm1}. Without loss of generality, let $R=\{v_1,v_2, \ldots , v_{n-2}\}$.
%\end{proof}

\begin{claim}\label{claim1}
 $G[R]$ is a disjoint union of some paths.
\end{claim}
 \noindent \emph{Proof.}
% \begin{proof}
 We only need to prove that $G[R]$ does not contain cycles and $d_R(v_i) \le 2$ for all $i \in \{1,2, \ldots ,n-2\}$. We first assume that $G[R]$ contains a cycle $C_m=v_1v_2 \cdots v_{m}v_1$, where $3 \leq m \leq n-2$.  If $3 \leq m<n-2$, then there exists an $i \in \{m+1,  \ldots  ,n-2\}$ such that $v_i \notin  V (C_m)$. Without loss of generality, suppose that $i=n-2$. Then $G[V(C_m) \cup \{u',u'',v_{n-2}\}]$ contains a $K_5$-minor. This contradicts the fact that $G$ is a planar graph. If $m=n-2$, then for sufficiently large $n > l+k-1 $, $G[R]$ contains $P_{l+k-3}=v_1v_2 \cdots v_{l+k-3}$. Thus, the union of the cycles $u'v_{1} \cdots v_{l-1}u'$ and $u'v_l \cdots v_{l+k-4}u''v_{l+k-3}u'$ forms a copy of $C_{k,l}$ in $G$, a contradiction. Therefore, we get $G[R]$ dose not contain cycles. Next, we assume that there exists a vertex $v_i$ satisfying $d_R(v_i) \geq 3$. Thus, $|N_{R}(v_i)| \ge 3$ and denote $\{v_{i_1}, v_{i_2}, v_{i_3}\} \subset N_R(v_i)$. Clearly, $G[\{u', u’'', v_{i}, v_{i_1}, v_{i_2}, v_{i_3}\}]$ contains $K_{3,3}$, again contradicting planarity of $G$. Thus, $d_R(v_i) \le 2$ for all $i \in \{1,2, \ldots ,n-2\}$. We conclude that $G[R]$ is a disjoint union of some paths.
 	%	\end{proof}
		
\begin{claim}\label{claim2}
$u'u'' \in E(G)$.
\end{claim}
 \noindent \emph{Proof.} 
 %\begin{proof}
Suppose to the contrary that $u'u''\notin E(G)$. Let $G'=G+{u'u''}$. Clearly, $G'$ is a planar graph. We assert that $G'$ remains $C_{k,l}$-free. Otherwise, suppose that $G'$ contains $C_{k,l}$. Then $u'u''\in E(C_{k,l})$. Without loss of generality, assume that $u'u'' \in E(C_l)$. Let $C_l=u'u''v_1v_2 \cdots v_{l-2}u'$. If $V(C_k) \cap V(C_l)=\{u'\}$ or $\{u''\}$, by symmetry, it suffices to consider the  $V(C_k) \cap V(C_l) = \{u'\}$. Let $C_k = u'v_{l-1}\cdots v_{l+k-3}u'$. Then, the union of the cycles $u'v_{l-1} \cdots v_{k+l-3}u'$ and $u'v_1u''v_2 \cdots v_{l-2}u'$ forms a copy of $C_{k,l}$ in $G$, a contradiction. If $V(C_k) \cap V(C_l)=\{v_i\}$ for some $i \in \{1,2, \ldots ,l-2\}$, then $G$ contains a cylce $C_{k}$ in $G[R]$, which contradicts Claim \ref{claim1}. Therefore, $G'$ is a planar graph without $C_{k,l}$. Since $G \subsetneq G'$, by Lemma \ref{lm4}, $\rho(G') > \rho(G)$, contradicting the extremality of $G$. Thus, $u'u'' \in E(G)$.

%\begin{case}
%$V(C_k) \cap V(C_l)=\{u'\}$ or $\{u''\}$.

%By symmetry, it suffices to consider the case $V(C_k) \cap V(C_l) = \{u'\}$. Let $C_k = u'v_{l-1}\cdots v_{l+k-3}u'$. Then $u'v_{l-1} \cdots v_{k+l-3}u'$ and $u'v_1u''v_2 \cdots v_{l-2}u'$ are two cycles of length $k$ and $l$.
%These two cycles share the common vertex $u'$, so $G$ contains $C_{k,l}$, a contradiction.
%\end{case}
%Let $C_k=u'v_{l-1}v_l \cdots v_{k+l-3}u'$. Obviously, $u'v_{l-1}v_l \cdots v_{k+l-3}u'$ and $u'v_1u''v_2 \cdots v_{l-2}u'$ are two cycles of length $k$ and $l$, and they intersect at $u'$. Then $G$ contains $C_{k,l}$ as a subgraph, this is a contradiction.

%\begin{case}
%$V(C_k) \cap V(C_l)=\{u''\}$. 
%Let $C_k=u''v_{l-1}v_l \cdots v_{k+l-3}u''$. Obviously, $u''v_{l-1}v_l \cdots v_{k+l-3}u''$ and $u''v_1u'v_2 \cdots v_{l-2}u''$ are two cycles of length $k$ and $l$, and they intersect at $u''$. Then $G$ contains $C_{k,l}$ as a subgraph, this is a contradiction.
%\end{case}

%\begin{case}
%$V(C_k) \cap V(C_l)=\{v_i\}$, where $i \in \{1,2, \cdots ,l-2\}$.

%Without loss of generality, assume that $V(C_k) \cap V(C_l)=\{v_{l-2}\}$. Let $C_k=v_{l-2}v_{l-1} \cdots$ $v_{l+k-3}$ $v_{l-2}$.
%Obviously, $G[\{u',u'',v_{l-2},v_{l-3},v_{l-1},v_{l+k-3}\}]$ form a copy of $K_{3,3}$-minor. This contradicts the fact that G is a planar graph.
%\end{case}
%By combining the proofs of Case $1$ and Case $2$, we get that $G'$ is a planar graph without $C_{k,l}$. Since $G \subsetneq G'=G+{u'u''}$, by Lemma \ref{lm4}, $\rho(G') > \rho(G)$. This contradicts the definition of $G$. Thus, $u'u'' \in E(G)$.
%\end{proof}

Suppose that $H$ is a disjoint union of $q$ paths, denoted by $\bigcup_{i=1}^{q} P_{{n_i}(H)}$ with $n_1(H) \geq n_2(H) \geq \cdots \geq n_{q}(H)$, and $q \geq 2$. Let $K_2 + H$ be $C_{k,l}$-free. The following claims hold.

\begin{claim}\label{claim3} 
$n_1(H)+n_2(H) \leq l+k-4$.
\end{claim}
 \noindent \emph{Proof.} 
 %\begin{proof}
Suppose to the contrary that $n_1(H)+n_2(H) \geq l+k-3$. Let $P_{n_1(H)}=v_1v_2 \cdots v_{n_1(H)}$ and $P_{n_2(H)}=v_{n_1(H)+1}v_{n_1(H)+2} \cdots v_{n_1(H)+n_2(H)}$. If $n_1(H) \leq {l-3}$, then the union of the cycles $ u'v_1v_2 \cdots$ $ v_{n_1(H)}$ $u''v_{{n_1(H)}+1} \cdots$ $v_{l-2}u'$ and $u''v_{l-1}v_l \cdots $ $v_{l+k-3}u''$ forms a copy of $C_{k,l}$ in $G$ obviously, a contradiction. If $n_1(H)={l-2}$, then the union of the cycles $u'v_1v_2 \cdots v_{l-2}u''u'$ and $u'v_{l-1}v_{l} \cdots $ $v_{l+k-3}u'$ forms a copy of $C_{k,l}$ in $G$, a contradiction. If ${l-1} \leq {n_1(H)} \leq {l+k-4}$, then we get the union of the cycles $u'v_1v_2 \cdots v_{l-1}u'$ and $u'v_{l} \cdots v_{n_1(H)}u''v_{{n_1(H)}+1}$ $v_{{n_1(H)}+2} \cdots v_{l+k-3}u'$ which forms a copy of $C_{k,l}$ in $G$, a contradiction. If $n_1(H) \geq {l+k-3}$, then the union of the cycles $u'v_1v_2 \cdots v_{l-1}u'$ and $u'v_{l}v_{l+1} $ $\cdots v_{l+k-3}u''u'$ also forms a copy of $C_{k,l}$ in $G$, a contradiction. Therefore, $n_1(H)+n_2(H) \leq l+k-4$.

\begin{claim}\label{claim4} 
If $H$ satisfy the following three conditions:
 
 	\begin{description}
 	\item[(i)]$n_1(H)+n_2(H) \leq l+k-4$;
 	\item[(ii)]$n_1(H)+n_2(H)+n_3(H) \leq l+k-4$ or there are no three paths with one of order at least $l-1$ and the sum of the other two is at least $k-2$;
 	\item[(iii)]$n_1(H)+n_2(H)+n_3(H) \leq l+k-4$ or there are no three paths with one of order at least $k-1$ and the sum of the other two is at least $l-2$.
 	\end{description}
Then ${K_2+H}$ is $C_{k,l}$-free.
 \end{claim}
  \noindent \emph{Proof.} 
 %\begin{proof}
We prove the contrapositive: if $C_{k,l} \subset K_2 + H$, then at least one of $(i), (ii), (iii)$ is violated. And one of the following cases holds.
 \begin{caseee}
 $V(C_l) \cap V(C_k)=\{u'\}$ or $V(C_l) \cap V(C_k)=\{u''\}$.
 
 \end{caseee}
 
 By symmetry, we only consider $V(C_l) \cap V(C_k)=\{u'\}$.

 \begin{subcaseee}
 	
 $u'' \in V(C_k)$, $u'' \notin V(C_l)$.

 Let $C_l=u'v_1 \cdots v_{l-1}u'$. If  $u'u'' \in E(C_k)$, let $C_k=u'v_l \cdots v_{l+k-3}u''u'$, then $H$ contains $P_{l-1} \cup P_{k-2}=v_1v_2 \cdots v_{l-1} \cup v_lv_{l+1} \cdots v_{l+k-3}$ or $P_{l+k-3}=v_1v_2 \cdots v_{l-1}v_lv_{l+1} \cdots v_{l+k-3}$. Therefore, $n_1(H)+n_2(H) \geq \min\{l-1+k-2,l+k-3+1\}=l+k-3$. This contradicts condition $(i)$.
 Next, assume $u'u'' \notin E(C_k)$, and $C_k=u'v_lv_{l+1} \cdots v_{l+a}u''v_{l+a+1} \cdots v_{l+k-3}u'$. 
 If $\{v_{l-1}v_l ,v_{l+a}v_{l+a+1}\} \subset E(H)$, then $H$ contains $P_{l+k-3}$. Therefore, $n_1(H)+n_2(H) \geq l+k-2$. This contradicts condition $(i)$. 
 If $v_{l-1}v_l \in E(H)$ and  $v_{l+a}v_{l+a+1} \notin E(H)$, then $H$ contains $P_{l+a} \cup P_{k-a-3}=v_1v_2 \cdots v_{l-1}$ $v_lv_{l+1} \cdots v_{l+a} \cup$ $v_{l+a+1}$ $ \cdots v_{l+k-3}$. Therefore, $n_1(H)+n_2(H) \geq l+k-3$. This contradicts condition $(i)$.
 If $v_{l-1}v_l \notin E(H)$ and $v_{l+a}v_{l+a+1} \in E(H)$, then $H$ contains $P_{l-1} \cup P_{k-2}=v_1v_2 \cdots v_{l-1} \cup v_lv_{l+1} \cdots v_{l+a}v_{l+a+1} \cdots v_{l+k-3}$. Therefore, $n_1(H)+n_2(H) \geq l+k-3$. This contradicts condition $(i)$.
 If $\{v_{l-1}v_l,v_{l+a}$ $v_{l+a+1}\} \not\subset E(H)$, then $H$ contains $P_{l-1} \cup P_{a+1} \cup P_{k-a-3}=v_1v_2 \cdots v_{l-1} \cup v_lv_{l+1} \cdots v_{l+a} \cup v_{l+a+1} \cdots v_{l+k-3}$. Therefore, $n_1(H)+n_2(H)+n_3(H) \geq l+k-3$ and one path has order at least $l-1$ and the sum of the other two is at least $k-2$. This contradicts condition $(ii)$.
 
 \end{subcaseee}
 \begin{subcaseee}
 $u'' \in V(C_l)$ and $u'' \notin V(C_k)$ .
 
Let $C_k=u'v_1v_2 \cdots v_{k-1}u'$. If $u'u'' \in E(C_l)$, let $C_l=u'v_{k} \cdots v_{k+l-3}u''u'$, then $H$ contains $P_{k-1} \cup P_{l-2}=v_1v_2 \cdots v_{k-1} \cup v_{k}v_{k+1} \cdots v_{l+k-3} $ or $P_{l+k-3}=v_1v_2 \cdots v_{l+k-3} $. Therefore, $n_1(H)+n_2(H) \geq \min\{l-2+k-1,l+k-3+1\}=l+k-3$. This contradicts condition $(i)$.
Next, assume $u'u'' \notin E(C_l)$,  and $C_l=u'v_kv_{k+1} \cdots v_{k+a}u''v_{k+a+1} \cdots v_{l+k-3}u'$. If $\{v_{k-1}v_k,$ $v_{k+a}$ $v_{k+a+1}\}$ $\subset E(H)$, then $H$ contains $P_{l+k-3}$. Therefore, $n_1(H)+n_2(H) \geq l+k-3+1=l+k-2$. This contradicts condition $(i)$.
 If $v_{k-1}v_k \in E(H)$ and $v_{k+a}v_{k+a+1} \notin E(H)$, then $H$ contains $P_{k+a} \cup P_{l-a-3}=v_1v_2 \cdots v_{k-1}$ $v_kv_{k+1} \cdots v_{k+a} \cup v_{k+a+1} \cdots v_{l+k-3}$. Therefore, $n_1(H)+n_2(H) \geq l+k-3$. This contradicts condition $(i)$.
 If $v_{k-1}v_k \notin E(H)$ and $v_{k+a}v_{k+a+1} \in E(H)$, then $H$ contains  $P_{k-1} \cup P_{l-2}=v_1v_2 \cdots v_{k-1} \cup v_kv_{k+1} \cdots v_{k+a}v_{k+a+1} \cdots$ $ v_{l+k-3}$. Therefore, $n_1(H)+n_2(H) \geq l+k-3$. This contradicts condition $(i)$.
 If $\{v_{k-1}v_k,v_{k+a}$ $v_{k+a+1}\} \not\subset E(H)$, then $H$ contains $P_{k-1} \cup P_{a+1} \cup P_{l-a-3}=v_1v_2 \cdots v_{k-1} \cup v_kv_{k+1} \cdots v_{k+a} \cup v_{k+a+1} \cdots v_{l+k-3}$. Therefore, $n_1(H)+n_2(H)+n_3(H) \geq l+k-3$ and one path has order at least $k-1$ and the sum of the other two is at least $l-2$. This contradicts condition $(iii)$.

 	  \end{subcaseee}
   \begin{subcaseee}
   	
$u'' \notin V(C_{k,l})$.

 	Let $C_l=u'v_1v_2 \cdots v_{l-1}u'$, $C_k=u'v_{l}v_{l+1} \cdots v_{l+k-2}u'$. Thus, $H$ contains $P_{l-1} \cup P_{k-1}=v_1v_2 \cdots v_{l-1}$ $ \cup v_{l}v_{l+1} \cdots v_{l+k-2} $ or $P_{l+k-2}=v_1v_2 \cdots v_{l-1}v_{l}v_{l+1} \cdots v_{l+k-3}$ . Therefore, $n_1(H)+n_2(H) \geq \min\{l-1+k-1,l+k-2+1\}=l+k-2$. This contradicts condition $(i)$.
 	
\end{subcaseee}

   \begin{caseee}

$V(C_l) \cap V(C_k)=\{v\} \subset V(H)$.

 	Since $H$ dose not contain cycles and $K_2 + H$ contains $C_{k,l}$. Without loss of generality, we assume that $u' \in V(C_l)$, $u'' \in V(C_k)$. Let $C_l=u'v_1v_2 \cdots v_{l-2}vu'$, $C_k=u''vv_{l-1} \cdots v_{l+k-4}u''$. Thus, $H$ contains $P_{l+k-3}=v_1v_2 \cdots v_{l-2}vv_{l-1}v_{l} \cdots v_{l+k-4}$. Therefore, $n_1(H)+n_2(H) \geq l+k-2$. This contradicts condition $(i)$.
 	
 	 \end{caseee}
 	 
 %\end{proof}
  \vspace*{2mm}
  
Recall from Claims \ref{claim1} and \ref{claim2} that $G \cong K_2 +H$ , where $H \cong G[R] \cong \bigcup_{i=1}^{q} P_{n_i}$ with $n_1 \ge n_2 \ge \ldots \ge n_q$, and $u'u'' \in E(G)$. For convenience, let $P_{n_1}=v_1v_2 \cdots v_{n_1}$, $P_{n_2}=v_{n_1+1}  \cdots $ $v_{n_1+n_2}$ and $P_{n_3}=v_{n_1+n_2+1} \cdots v_{n_1+n_2+n_3}$.

\begin{claim}\label{claim5} 
 For integers $3 \leq  k < l \leq {2k-2}$ , $n >\max\{ \frac{21k^3 - 149k^2 + 338k - 240}{4}, \frac{187}{2}+9\sqrt{102}, 2.67 \times 9^{17}, \frac{10}{9}|V(F)|,10.2\times2^{k-2}+2\}$ , then $\rho(G) \leq \rho(K_2+H_{\mathcal{P}}(k-2,k-2))$, with equality if and only if $G \cong K_2+H_{\mathcal{P}}(k-2,k-2)$.

\end{claim}
  \noindent \emph{Proof.} We split into $k=3$ and $k\ge 4$. When $k=3$, by $k+1\le l\le 2k-2$, we have $l=4$. In this case, we claim that $G \cong K_2+H_{\mathcal{P}}(1,1)$, i.e., $e(H)=0$. Otherwise, suppose that there exists an edge $v_1v_2\in E(H)$, i.e., $P_{n_1}=v_1v_2$, and $v_3,v_4$ be two other vertices in $V(H)$ different from $v_1,v_2$. Obviously, the union of the cycles $u'v_1v_2u'$ and $u'v_3u''v_4u'$ forms a copy of $C_{3,4}$ in $G$, a contradiction. Consequently, we only need to discuss the case $k\ge 4$.

We first claim that $n_3 \ge k-2$. Otherwise, suppose that $n_3 < k-2$.
Let $G' \cong K_2+H_{\mathcal{P}}(k-2,k-2)$, by Claim \ref{claim4}, $G'$ is $C_{k,l}$-free. The graph $G'$ is obtained from $G$ by deleting at most $a_1=e(P_{n_1})+e(P_{n_2})-(k-3+k-3) \leq k-2$ edges and adding at least $b_1= \left\lfloor \frac{n-2k+2}{k-2} \right\rfloor (k-3) - \left\lceil \frac{n-({n_1}+{n_2})-2}{n_3} \right\rceil (n_3-1)$ edges. Thus,

\[
\begin{aligned}  
	b_1 &\geq \left( \frac{n-2k+2}{k-2}-1 \right)(k-3) - \left( \frac{n-({n_1}+{n_2})-2}{n_3}+1 \right)(n_3-1)\\
		&= \frac{(k-2-{n_3})n+(4-3k)(k-3){n_3}+({n_1}+
		{n_2}+2)(n_3-1)(k-2)-n_3(n_3-1)(k-2)}{(k-2)n_3}\\
	&\geq \frac{n+(4-3k)(k-3)(k-3)-(k-3)(k-4)(k-2)}{(k-2)(k-3)}\quad \text{(since $n_3 \le k-3$ )}\\
	&= \frac{n - 4k^3 + 31k^2 - 77k + 60}{(k-2)(k-3)}.
\end{aligned}
\]
 By Lemma \ref{lm2} and the Rayleigh principle, 
\[
\begin{aligned}  
	\rho(G')-\rho(G) &\geq \frac{X^T(A(G') - A(G))X}{X^TX}
	\geq \frac{{2b_1}\frac{4}{\rho^2}-{2a_1}\frac{5}{\rho^2}}{X^TX}\\
	&\geq \frac{2}{X^TX} \left( \frac{4(n - 4k^3 + 31k^2 - 77k + 60)}{(k-2)(k-3)\rho^2}-\frac{5(k-2)}{\rho^2} \right) \\
	&= \frac{2}{X^TX} \left(\frac{4n - 21k^3 + 159k^2 - 388k + 300}{(k-2)(k-3)\rho^2}\right)\\
	&>0,
\end{aligned}
\]
where $n >\max\{ \frac{21k^3 - 159k^2 + 388k - 300}{4}, \frac{187}{2}+9\sqrt{102}\}$ and $\rho=\rho(G) \geq \sqrt{2n-4}$. Hence  $\rho(G) < \rho(G')$. By the extremality of $G$, this is a contradiction. 

%We first assert that $n_3 \leq {k-2}$. Otherwise, suppose that $n_3 \geq {k-1}$. Then $n_1 \ge n_2 \ge n_3 \ge k-1$, hence $n_1 + n_2 \ge 2k-2 \ge l$. Thus, $G[V(P_{n_1}) \cup V(P_{n_2}) \cup V(P_{n_3}) \cup \{u', u''\}]$ contains $C_{k,l}$, a contradiction. Therefore, $n_3 \leq k-2$.  If $n_3\leq k-3$, a contradiction also follows from Claim \ref{claim5}. Consequently, we have $n_3=k-2$.

%Next, we show that $n_2 \leq {k-2}$. Otherwise, suppose that $n_2 \geq {k-1}$. Then $n_1 \geq n_2 \geq k-1,$ this implies both $G[V(P_{n_1}) \cup \{u'\}]$ and $G[V(P_{n_2}) \cup \{u'\}]$ contain $k$-cycle. We assert that $n_2+n_3 \leq n_1+n_3 \leq {l-3}$. If not, then $n_1+n_3 \geq n_2+n_3 \geq {l-2}$. Let $P_{n_1}=v_1v_2 \cdots v_{n_1}$, $P_{n_3}=v_{{n_1}+1} \cdots v_{l-2} \cdots v_{{n_1}+{n_3}}$. Therefor, $u'v_1v_2 \cdots v_{n_1}u''$ $v_{{n_1}+1} \cdots v_{l-2}u'$ is an $l$-cycle. It is straightforward to see that $G[V(P_{n_1} \cup P_{n_2} \cup P_{n_3} \cup {u',u''})]$ contains $C_{k,l}$, a contradiction. Thus, $n_3 \leq {l-3-{n_2}} < {k-3}$. By Claim \ref{claim5}, $\rho(G) < \rho(K_2 + H_{\mathcal{P}}(k-2,k-2))$, a contradiction. 
Next, we claim that $n_1 \leq {k-2}$. Suppose to contrary that $n_1 \geq {k-1}$. If $n_2+n_3 \leq {l-3} \le 2k-5$, then $n_3 \le \left\lfloor \frac{2k-5}{2} \right\rfloor =k-3 $, a contradiction. If $n_2+n_3 \geq {l-2}$, then the union of the cycles $u'v_1v_2 \cdots v_{k-1}u'$ and $u'v_{n_1+1}  \cdots v_{n_1+n_2}u''v_{n_1+n_2+1}  \cdots $ $ v_{n_1+l-2}u'$ forms a copy of $C_{k,l}$ in $G$, a contradiction. Therefore, $n_1 \le k-2$.

Finally, we prove that $n_i ={k-2}$ for $i \in \{1,2, \ldots ,q-1\}$. Otherwise, suppose there exist an $i_0=\min\{i|1 \leq i \leq q-1,n_i \leq {k-3}\}$. Let $H_1$ be an $(n_{i_0},n_q)$-transformation of $H$. Thus, $n_1(H_1)=\max\{n_1,n_{i_0}+1\} \leq {k-2}$. Since ${k-2} \geq n_1(H_1) \geq n_2(H_1) \geq n_3(H_1)$, $K_2+H_1$ is $C_{k,l}$-free and $\rho(K_2+H_1)>\rho(G)$ by Claim \ref{claim4} and Lemma \ref{lm3}, a contradiction. Thus, $n_i ={k-2}$ for $i \in \{1,2, \ldots ,q-1\}$. Therefore, $G \cong K_2+H_{\mathcal{P}}(k-2,k-2)$.

%\end{proof}
\vspace*{2mm} 
 \begin{claim}\label{claim6} 
 For integers $l \in \{2k-1,2k\}$, $n > \max \{5k^3-38k^2+92k-69,2.67 \times 9^{17}, \frac{10}{9}|V(F)|,$ $10.2\times2^{k-2}+2\}$ , then $\rho(G) \leq \rho(K_2+H_{\mathcal{P}}(l-2,k-2))$, with equality if and only if $G \cong K_2+H_{\mathcal{P}}(l-2,k-2)$.
 	
 \end{claim}
  \noindent \emph{Proof.} Again split into $k=3$ and $k\ge 4$. For $k=3$, we have $l=2k-1=5$. Then we claim that $G \cong  K_2+H_{\mathcal{P}}(3,1)$, i.e., $e(H)=2$. Otherwise, suppose that $e(H) \ge 3$. When $e(H)=3$, three distinct cases arise. If $e(P_{n_1}) = 1, e(P_{n_2}) = 1, e(P_{n_3}) = 1$, then the union of the cycles $u'v_1v_2u'$ and $u'v_{3}v_{4}u''v_{5}u'$ forms a copy of $C_{3,5}$ in $G$, a contradiction. If $e(P_{n_1}) = 2, e(P_{n_2}) = 1$, then the union of the cycles $u'v_1v_2v_3u''u'$ and $u'v_{4}v_{5}u'$ forms a copy of $C_{3,5}$ in $G$, which also gives a contradiction. If $e(P_{n_1}) = 3$, then  $G[V(P_{n_1})\cup\{u',u'',v_{5}\}]$ contains $C_{3,5}$, a contradiction. For $e(H)\ge 4$, $H$ contains a subgraph with three edges of one of the above forms, giving the same contradiction. Consequently, we obtain $e(H)\leq 2$. Since $\rho(K_2+H_{\mathcal{P}}(3,1))>\rho(K_2+H_{\mathcal{P}}(2,1))$ and $K_2+H_{\mathcal{P}}(3,1)$ is $C_{3,5}$-free, we have $e(H)=2$. Next, we claim that $G\cong K_2+H_{\mathcal{P}}(3,1)$. If not, then $H_{\mathcal{P}}(3,1)$ can be transformed from $H$ through at most two $(s_1,s_2)$-transformations. By Lemma \ref{lm3}, $\rho(K_2+H_{\mathcal{P}}(3,1))>\rho(G)$, a contradiction. Similarly, if $l=2k=6$, then $G \cong K_2+H(4,1)$. Therefore, we only consider the case $k\geq 4$ in the following discussion.

We first claim that $n_3 \ge k-2$. Suppose to contrary that $n_3 < k-2$. Let $G'=K_2+H_{\mathcal{P}}(l-2,k-2)$, by Claim \ref{claim4}, $G'$ is $C_{k,l}$-free. $G'$ is obtained from $G$ by deleting at most $a_2=e(P_{n_1})+e(P_{n_2})-(l-3+k-3) \leq 0$ edges and adding at least $b_2= \left\lfloor \frac{n-3k+3}{k-2} \right\rfloor (k-3) - \left\lceil \frac{n-({n_1}+{n_2})-2}{n_3} \right\rceil (n_3-1)$ edges. Thus,
\[
\begin{aligned}  
b_2 &\geq \left( \frac{n-3k+3}{k-2}-1 \right)(k-3) - \left( \frac{n-({n_1}+{n_2})-2}{n_3}+1 \right)(n_3-1)\\
&= \frac{(k-2-{n_3})n+(5-4k)(k-3){n_3}+({n_1}+
	{n_2}+2)(n_3-1)(k-2)-n_3(n_3-1)(k-2)}{(k-2)n_3}\\
&\geq \frac{n+(5-4k)(k-3)(k-3)-(k-4)(k-2)(k-3)}{(k-2)(k-3)}\quad \text{(since $n_3 \le k-3$ )}\\
&= \frac{n-5k^3+38k^2-92k+69}{(k-2)(k-3)}.
\end{aligned}
\]
By Lemma \ref{lm2} and the Rayleigh principle, 
\[
\begin{aligned}  
\rho(G')-\rho(G) &\geq \frac{X^T(A(G') - A(G))X}{X^TX}
\geq \frac{{2b_2}\frac{4}{\rho^2}}{X^TX}\\
&\geq \frac{2}{X^TX} \left( \frac{4(n-5k^3+38k^2-92k+69)}{(k-2)(k-3)\rho^2} \right) \\
&>0,
\end{aligned}
\]
where $n > 5k^3-38k^2+92k-69$ and $\rho=\rho(G) \geq \sqrt{2n-4}$. So $\rho(G) < \rho(G')$, a contradiction by the extremality of $G$. 
%\end{Proof}

We next assert that $n_3 = {k-2}$.  Suppose for contradiction  $n_3 \geq k-1$. Thus, $n_1 \geq n_2 \geq n_3 \geq {k-1}$ and $n_2+n_3 \geq {2k-2} \geq l-1$. Obviously, the union of the cycles $u'v_1v_2 \cdots v_{k-1}u'$ and $u'v_{n_1+1}  \cdots v_{n_1+n_2}u''v_{n_1+n_2+1}  \cdots $ $ v_{n_1+l-2}u'$ forms a copy of $C_{k,l}$ in $G$, a contradiction. 
 
We then claim that $n_1 \leq l-2$. Otherwise, suppose that $n_1 \geq l-1$. If $n_2 +n_3 \le k-3$, then $n_3 \le k-3$, a contradiction. If $n_2 +n_3 \ge k-2$, then the union of the cycles $u'v_1v_2 \cdots v_{l-1}u'$ and $u'v_{n_1+1}  \cdots $ $v_{n_1+n_2}u''v_{n_1+n_2+1} \cdots v_{n_1+k-2}u'$ forms a copy of $C_{k,l}$ in $G$, a contradiction. Therefore, $n_1 \le l-2$.

%$P_{n_2}=v_{n_1+1}  \cdots $ $v_{n_1+n_2}$, $P_{n_3}=v_{n_1+n_2+1} \cdots v_{n_1+n_2+n_3}$. 
 %We claim that $n_1 \leq l-2$. Otherwise, suppose that $n_1 \geq l-1$. And $n_1 \leq {l+k-5}$ by Claim \ref{claim3}. Let $P_{n_1}=v_1v_2 \cdots v_{l-1} \cdots v_{n_1}$. We assert that $n_3 \leq n_2 \leq {l+k-4-{n_1}} \leq {k-3}$. If not, then $n_2 \geq {k-2}$. Let $P_{n_2}=u_1u_2 \cdots u_{k-2} \cdots u_{n_2}$. It is straightforward to see that $u'v_1v_2 \cdots v_{l-1}u'$ and $u'u_1u_2 \cdots u_{k-2}u''u'$ are two cycles of length $l$ and $k$, and they intersect  at $u'$. Then $G$ contains $C_{k,l}$, a contradiction. Therefore, $n_3 \leq n_2 \leq k-3$. By Claim \ref{claim6}, $\rho(G) < \rho(K_2 + H_{\mathcal{P}}(l-2,k-2))$, a contradiction.
%Next, we assert that $n_2 \leq k-2$. Suppose to the contrary that $n_1 \geq n_2 \geq k-1$. 
 %If $n_1 + n_3 \le l-3$.  When $l=2k-1$, then $n_3 \le l-3 - n_1 \le k-3$, a contradiction. When $l=2k$, then $n_3 \leq {l-3-{n_1}} \leq {k-2}$. If $n_3=k-2$, then $n_1 \leq {l-3-(k-2)}={k-1}$. Therefore, $n_1=n_2=k-1$. Let $H_2$ be the $(n_1,n_2)$-transformation of $H$. Then $n_1(H_2)=k,n_2(H_2)=n_3(H_2)=k-2$, $K_2+H_2$ is $C_{k,l}$-free and $\rho(K_2+H_2)>\rho(G)$ by Claim \ref{claim4} and Lemma \ref{lm3},  a contradiction. Thus, $n_3 \leq {k-3}$, a contradiction. If $n_1 + n_3 \ge l-2$, then the union of the cycles $u'v_{n_1+1}v_{n_1+2} \cdots v_{n_1+k-1}u'$ and $u'v_1v_2  \cdots v_{n_1}u''v_{n_1+n_2+1}  \cdots $ $ v_{n_2+l-2}u'$ forms a copy of $C_{k,l}$ in $G$, a contradiction. Therefore, $k-2 \ge n_2 \ge n_3=k-2$.
Next, we assert that $n_2 \leq k-2$. Suppose to the contrary that $n_1 \geq n_2 \geq k-1$. We split the argument into two cases based on the value of $n_1 + n_3$.
\begin{casee}
$n_1 + n_3 \leq l-3$.
\end{casee} 
We further separate the two possible values of $l$.
\begin{subcasee}
$l = 2k-1$.
\end{subcasee}
 Clearly, $n_3 \leq l-3 - n_1 \leq  k-3$, a contradiction.
 \begin{subcasee}
 $l = 2k$.
 \end{subcasee}
  We have $ n_3 \leq l-3 - n_1 \leq k-2.$
 If $n_3 = k-2$, then $n_1 \leq l-3 - n_3 = (2k-3) - (k-2) = k-1.$
 Combined with our contradictory assumption $n_1 \geq k-1$, we have $n_1 = n_2 = k-1$. Let $H_2$ be the $(n_1,n_2)$-transformation of $H$. Then $n_1(H_2) = k, n_2(H_2) = n_3(H_2) = k-2.$ By Claim \ref{claim4} and Lemma \ref{lm3}, $K_2 + H_2$ is $C_{k,l}$-free and $\rho(K_2 + H_2) > \rho(G)$. This contradicts the extremality of $G$. If $n_3 \leq k-3$, we also obtain a contradiction. 
 \begin{casee}
$n_1 + n_3 \geq l-2$.
 \end{casee}
Then the union of the cycles $u'v_{n_1+1}v_{n_1+2} \cdots v_{n_1+k-1}u'$ and $u'v_1v_2  \cdots v_{n_1}u''v_{n_1+n_2+1}  \cdots $ $ v_{n_2+l-2}$ $u'$ forms a copy of $C_{k,l}$ in $G$, a contradiction.

 %Next, we assert that $n_2 \leq k-2$. Otherwise, $n_1 \geq n_2 \geq k-1$, this implies both $G[V(P_{n_1}) \cup \{u'\}]$ and $G[V(P_{n_2}) \cup \{u'\}]$ contain $k$-cycle as a subgraph. Thus, $n_2 + n_3 \le n_1 + n_3 \le l-3$. If not, then $n_1 + n_3 \ge n_2 + n_3 \ge l-2$. Let $P_{n_1} = v_1 v_2 \cdots v_{n_1}$, $P_{n_3} = v_{n_1+1} \cdots v_{l-2} \cdots v_{n_1+n_3}$, then $u' v_1 v_2 \cdots v_{n_1} u'' v_{n_1+1} \cdots $ $v_{l-2} u'$ is an $l$-cycle, a contradiction. When $l=2k-1$, then $n_3 \le l-3 - n_1 \le 2k-1-3-(k-1)=k-3$. By Claim \ref{claim6}, $\rho(G) < \rho(K_2 + H_{\mathcal{P}}(l-2,k-2))$, a contradiction. 
 %When $l=2k$, then $n_3 \leq {l-3-{n_1}} \leq {k-2}$. If $n_3=k-2$, then $n_1 \leq {l-3-(k-2)}={k-1}$ and ${k-1} \leq {n_2} \leq {n_1} \leq {k-1}$. Therefore, $n_1=n_2=k-1$. Let $H_2$ be the $(n_1,n_2)$-transformation of $H$. Then $n_1(H_2)=k,n_2(H_2)=n_3(H_2)=k-2$, $K_2+H_2$ is $C_{k,l}$-free and $\rho(K_2+H_2)>\rho(G)$, by Claim \ref{claim4} and Lemma \ref{lm3},  a contradiction. Thus, $n_3 \leq {k-3}$. By Claim \ref{claim6}, $\rho(G) < \rho(K_2 + H_{\mathcal{P}}(l-2,k-2))$, a contradiction. As a consequence, when $l = 2k-1$ or $l = 2k$, it follows that $n_2 \le k-2$. Since $k-2\ge n_2\ge n_3=k-2$, we obtain $n_2=n_3=k-2$.
 
 Now we prove that $n_1+n_2={l+k-4}$. Otherwise, suppose that $n_1+n_2 \leq {l+k-5}$, and then $n_1 \leq {l-3}$ or $n_2 \leq {k-3}$. If $n_1 \leq {l-3}$, then let $H_3$ be the $(n_1,n_q)$-transformation of $H$. Therefore, $n_1(H_3)={n_1}+1 \leq {l-2},n_2(H_3)=n_2 \leq {k-2}$ and $n_3(H_3)=n_3 \leq {k-2}$. Clearly, by Claim \ref{claim4} and Lemma \ref{lm3}, $K_2+H_3$ remains $C_{k,l}$-free and $\rho(K_2+H_3)>\rho(G)$, a contradiction. If $n_2 \leq {k-3}$, then let $H_4$ be the $(n_2,n_q)$-transformation of $H$. Therefore, $n_1(H_4)=\max\{n_1,{n_2}+1\} \leq {l-2},n_2(H_4)=\min\{n_1,{n_2}+1\} \leq {k-2}$ and $n_3(H_4)=n_3 \leq {k-2}$. By Claim \ref{claim4} and Lemma \ref{lm3}, $K_2+H_4$ remains $C_{k,l}$-free and $\rho(K_2+H_4)>\rho(G)$  a contradiction. Thus, $n_1+n_2={l+k-4}$.

Finally, we prove that $n_i =n_2$ for $i \in \{3,4, \ldots ,q-1\}$. Suppose there exists an $i_0=\min\{i|4 \leq i \leq q-1,n_i \leq {{n_2}-1}\}$. Let $H_5$ be an $(n_{i_0},n_q)$-transformation. Thus, $n_1(H_5)=n_1 , n_2(H_5)=\max\{n_2,n_{i_0+1}\}=n_2 $ and $n_3(H_5)=n_3 $.  By Claim \ref{claim4} and Lemma \ref{lm3}, $K_2 + H_5$ is $C_{k,l}$-free and $\rho(K_2+H_5) > \rho(G)$ , a contradiction. Therefore, $G \cong K_2+H_{\mathcal{P}}(l-2,k-2)$.

%\end{Proof}
%\vspace*{2mm}
\begin{claim}\label{claim7} 
For inregers $l \geq {2k+1}$, $n > \max \{\frac{21l^3 - 255l^2 + 979l - 1145}{32} ,\frac{361{\left\lfloor \frac{l-3}{2} \right\rfloor}^2-722{\left\lfloor \frac{l-3}{2} \right\rfloor}-425}{32} ,2.67 \times 9^{17},\frac{187}{2}+9\sqrt{102}, \frac{10}{9}|V(F)|,10.2\times2^{\left\lfloor \frac{l-3}{2} \right\rfloor}+2\} $, then $\rho(G) \leq \rho(K_2+H_{\mathcal{P}}(\left\lceil \frac{l-3}{2} \right\rceil,\left\lfloor \frac{l-3}{2} \right\rfloor))$, with equality if and only if $G \cong K_2+H_{\mathcal{P}}(\left\lceil \frac{l-3}{2} \right\rceil,$ $\left\lfloor \frac{l-3}{2} \right\rfloor)$.
	
\end{claim}
\noindent \emph{Proof.}	  
%\begin{Proof}
We first claim that $n_3 \ge \left\lfloor \frac{l-3}{2} \right\rfloor $. Otherwise, suppose that $n_3 < \left\lfloor \frac{l-3}{2} \right\rfloor $.
%if $n_3 < \left\lfloor \frac{l-3}{2} \right\rfloor $, then $\rho(G) < \rho(K_2+H_{\mathcal{P}}(\left\lceil \frac{l-3}{2} \right\rceil,\left\lfloor \frac{l-3}{2} \right\rfloor))$. 
Let $G'=K_2+H_{\mathcal{P}}(\left\lceil \frac{l-3}{2} \right\rceil,\left\lfloor \frac{l-3}{2} \right\rfloor)$. Note that the longest cycle in $G'$ is $C_{l-1}$, so $G'$ is $C_{k,l}$-free. And $G'$ is obtained from $G$ by deleting at most $a_3=e(P_{n_1})+e(P_{n_2})-l+5 \leq k-1 $ edges and adding at least $b_3= \left\lfloor \frac{n-2-l+3}{\left\lfloor \frac{l-3}{2} \right\rfloor} \right\rfloor (\left\lfloor \frac{l-3}{2} \right\rfloor-1) - \left\lceil \frac{n-({n_1}+{n_2})-2}{n_3} \right\rceil (n_3-1)$ edges. Thus,
\[
\begin{aligned}  
	b_3 &\geq \left(  \frac{n-l+1}{\left\lfloor \frac{l-3}{2} \right\rfloor} -1 \right)(\left\lfloor \frac{l-3}{2} \right\rfloor-1) - \left( \frac{n-({n_1}+{n_2})-2}{n_3}+1 \right)(n_3-1)\\
	&\geq \frac{(\left\lfloor \frac{l-3}{2} \right\rfloor-{n_3})n+(1-l-\left\lfloor \frac{l-3}{2} \right\rfloor)(\left\lfloor \frac{l-3}{2} \right\rfloor -1){n_3}-n_3(n_3-1)\left\lfloor \frac{l-3}{2} \right\rfloor}{\left\lfloor \frac{l-3}{2} \right\rfloor n_3}\\
	&\geq \frac{n+(1-l-\frac{l-3}{2})(\frac{l-3}{2}-1)(\frac{l-5}{2})-(\frac{l-5}{2})(\frac{l-5}{2}-1)(\frac{l-3}{2})}{(\frac{l-3}{2})(\frac{l-5}{2})}\quad \text{(since $n_3 \le \frac{l-5}{2}$)}\\
	&\geq \frac{4n-2l^3+25l^2-98l+115}{(l-3)(l-5)}.
\end{aligned}
\]
By Lemma \ref{lm2} and the Rayleigh principle, 
\[
\begin{aligned}  
	\rho(G')-\rho(G) &\geq \frac{X^T(A(G') - A(G))X}{X^TX}
	\geq \frac{{2b_3}\frac{4}{\rho^2}-{2a_3}\frac{5}{\rho^2}}{X^TX}\\
	&\geq \frac{2}{X^TX} \left( \frac{4(4n-2l^3+25l^2-98l+115)}{(l-3)(l-5)\rho^2}-\frac{5(l-3)}{2\rho^2} \right) \\
	&= \frac{2}{X^TX} \left(\frac{ 32n - 21l^3 + 255l^2 - 979l + 1145}{2(l-3)(l-5)\rho^2} \right)\\
	&>0,
\end{aligned}
\]
where $n > \max \{\frac{21l^3 - 255l^2 + 979l - 1145}{32}, \frac{187}{2}+9\sqrt{102}\} $ and $\rho=\rho(G) \geq \sqrt{2n-4}$. Hence $\rho(G) < \rho(G')$. 
Similarly, establishing $n_3 < \left\lfloor \frac{l-3}{2} \right\rfloor $ yields a contradiction.

We now prove that $n_3 \geq {k-1}$. Suppose for contradiction that $n_3 \leq {k-2}$. Thus $n_3 \leq \frac{l-1}{2}-2 < \left\lfloor \frac{l-3}{2} \right\rfloor$, a contradiction. Thus, $n_1 \geq n_2 \geq n_3 \geq {k-1}$. 
Next, we claim that $n_1+n_2 \leq {l-3}$. Otherwise, suppose that $n_1+n_2 \geq {l-2}$. Then the union of the cycles $u'v_{n_1+n_2+1}v_{n_1+n_2+2} \cdots v_{n_1+n_2+k-1}u'$ and $u'v_1v_2  \cdots v_{n_1}u''v_{n_1+1}  \cdots v_{l-2}u'$ forms a copy of $C_{k,l}$ in $G$, a contradiction. Therefore, $n_1+n_2 \leq {l-3}$, and $n_2 \leq {\left\lfloor \frac{l-3}{2} \right\rfloor}$.

Next, we claim that $n_i =n_2$ for $i \in \{3,4, \ldots ,q-1\}$. Otherwise, suppose there exists an $i_0=\min\{i|3 \leq i \leq q-1,n_i \leq {{n_2}-1}\}$. If $i_0 > 3$,  then let $H_6$ be an $(n_{i_0},n_q)$-transformation of $H$. Thus, $n_1(H_6)=n_1,n_2(H_6)=\max\{n_2,n_{i_0+1}\}=n_2 $ and $n_3(H_6)=n_3$. Since $n_{1}(6)+n_2(6) \le l-3$, the longest cycle in $K_2+H_{6}$ is $C_{l-1}$. Then $K_2+H_{6}$ is $C_{k,l}$-free, and  $\rho(K_2+H_6) > \rho(G)$ by Lemma \ref{lm3}, a contradiction. If $i_0 \le 3$, then let $H_7$ be the $(n_{i_0},n_q)$-transformation of $H$. Thus, $n_{1}(H_7) =n_{1}$, $n_{2}(H_7) = n_{2}$ and $n_{3}(H_7)=\max\{n_{3},n_{i_0}+1\}=n_{3}.$ Using the methods similar to those for $i_0>3$, we can also get a contradiction. If $i_0=3$, let $H_{8}$ be the $(n_3,n_q)$-transformation of $H$. Then $n_{1}(H_{8})=n_{1},n_{2}(H_{8})=\max\{n_{2},n_{3}+1\}=n_{2},n_3(H_{8})={n_3}+1 $. Similarly, we get a contradiction.

Finally, we claim that $n_1+n_2=l-3 $ and $n_1=\left\lceil \frac{l-3}{2} \right\rceil, n_2=\left\lfloor \frac{l-3}{2} \right\rfloor$. If $n_1+n_2 \leq {l-4}$, then let $H_{9}$ be the $(n_1,n_q)$-transformation of $H$. Thus, $n_{1}(H_{9})=n_{1}+1,n_{2}(H_{9})=n_{2},n_3(H_{9})={n_3}$, and $n_1(9)+n_2(9) \leq {l-3}$. Therefore, $K_2+H_{9}$ is $C_{k,l}$-free, and then $\rho(K_2+H_{9}) > \rho(G)$ by Lemma \ref{lm3}, a contradiction. So $n_1+n_2=l-3$.
Next, we assume that $n_2 \leq \left\lfloor \frac{l-3}{2} \right\rfloor-1$, then $n_1 \geq \left\lceil \frac{l-3}{2} \right\rceil+1$. For sufficiently large $n$, let $n > n_1+(n_2+2){n_2}+2$ such that $q > n_2+4$. Then there exists at least $n_2+2$ paths of length $n_2$, denoted by $P^1,P^2, \ldots ,P^{{n_2}+2}$. We construct a new graph $G^*$ from $G$ via the following four edge operations:

 $(1)$ delete the edge $v_{\left\lceil \frac{l-3}{2} \right\rceil}v_{\left\lceil \frac{l-3}{2} \right\rceil+1}$ of $P_{n_1}$ and denote the two resulting paths by $P_{\left\lceil \frac{l-3}{2} \right\rceil}=v_1v_2 \cdots v_{\left\lceil \frac{l-3}{2} \right\rceil}$ and $P_{{n_1}-{\left\lceil \frac{l-3}{2} \right\rceil}}=v_{\left\lceil \frac{l-3}{2} \right\rceil+1}v_{\left\lceil \frac{l-3}{2} \right\rceil+2} \cdots v_{n_1}$;
 
 $(2)$ add an edge between an endpoint of $P^{n_2+2}$ and $v_{\left\lceil \frac{l-3}{2} \right\rceil+1}$ of $P_{n_1-\left\lceil \frac{l-3}{2} \right\rceil}$; 
 
 $(3)$ delete all the edge between each isolated vertex of $P^{n_2+1}$;
 
 $(4)$ add an edge between each isolated vertex of the ${n_2}$ isolated vertices and an endpoint of $P^j$ for each $j \in \{1,2, \ldots ,n_2\}$, respectivety.
	 
Actually, we delete $n_2$ edges from $G$ and add $n_2+1$ new edges. Therefore, $G^* \cong K_2+H^2$, where $H^2=P_{\left\lceil \frac{l-3}{2} \right\rceil} \cup P_{n_2+n_1-\left\lceil \frac{l-3}{2} \right\rceil} \cup n_2P_{n_2+1} \cup (q-n_2-4)P_{n_2} \cup P_{n_q}$. Thus, $n_1(H^2)+n_2(H^2)={l-3}$, and then $G^*$ is $C_{k,l}$-free. By Lemma \ref{lm2} and Rayleigh principle,
\[
\begin{aligned} 
\rho(G^*)-\rho(G) &\geq \frac{X^T(A(G^*) - A(G))X}{X^TX}\\
& \geq \frac{2{(n_2+1)}\frac{4}{\rho^2}-2{n_2}(\frac{4}{\rho^2}+\frac{18}{\rho^3}+\frac{21}{\rho^4})}{X^TX}\\
&\geq \frac{2{(n_2+1)}\frac{4}{\rho^2}-2{n_2}(\frac{4}{\rho^2}+\frac{19}{\rho^3})}{X^TX}\\
&\geq \frac{2}{X^TX} \left( \frac{4}{\rho^2}-\frac{19(\left\lfloor \frac{l-3}{2} \right\rfloor-1)}{\rho^3}\right)\\
&>0,
\end{aligned}
\]
where $n >\max\{\frac{361{\left\lfloor \frac{l-3}{2} \right\rfloor}^2-722{\left\lfloor \frac{l-3}{2} \right\rfloor}-425}{32},\frac{445}{2}\} $ and $\rho=\rho(G) \geq \sqrt{2n-4}$. Hence $\rho(G^*) > \rho(G)$, contradicting the maximality of $\rho(G)$. Therefore, $n_1=\left\lceil \frac{l-3}{2} \right\rceil, n_i=\left\lfloor \frac{l-3}{2} \right\rfloor$ for $i \in \{2,3, \dots , q-1 \}$, and then $G \cong K_2+H_{\mathcal{P}}(\left\lceil \frac{l-3}{2} \right\rceil,$ $\left\lfloor \frac{l-3}{2} \right\rfloor)$.

Combining the conclusions of Claims \ref{claim5}-\ref{claim7}, we complete the proof of Theorem \ref{thm1}.

%\end{Proof}
 \vspace*{2mm}
 \section{Proof of theorem \ref{thm2}}
In this section, based on the proofs for planar graphs, we determine the $\operatorname{SPEX}_{\mathcal{OP}}(n,C_{k,l})$ for $l \ge k \ge 3$ by proving Theorem \ref{thm2}. Before proceeding, we also list several useful lemmas, firstly.
%\vspace*{2mm} 
 \begin{lemma}\cite{W.H. Wang}.\label{lm8}  
 Let $F$ be an outerplanar graph not contained in $K_{1,n-1}$, where
 $n \geq \max\{400, $ $\frac{5}{4} |V(F)|\}$. Among the set of outerplanar graphs on $n$ vertices without $F$, the extremal graph having the maximum spectral radius contains a subgraph $K_{1,n-1}$. 
\end{lemma}
%\vspace*{2mm} 
\begin{definition}\cite{W.H. Wang}.\label{de4}
	Let $s_1$ and $s_2$ be two integers with $s_1\geq s_2\geq1$, and let $H=P_{s_1}\cup P_{s_2}\cup H_0$, where $H_0$ is a disjoint union of paths. We say that $H^*$ is an $(s_1,s_2)$-transformation of $H$ if
	$$H^*:=\begin{cases}P_{s_1+1}\cup P_{s_2-1}\cup H_0&\text{if }s_2\geq2,\\P_{s_1+s_2}\cup H_0&\text{if }s_2=1.\end{cases}$$
 Clearly, $H^*$ is a disjoint union of paths, which implies that $K_1 + H^*$ is an outerplanar graph.
 
\end{definition}
%\vspace*{2mm}
\begin{lemma}\cite{yin2026spectral}.\label{lm9}
	Let $G$ be a connected extremal graph with $\textit{spex}_{\mathcal{OP}}(n,F)$, where $F$ is an outerplanar subgraph of $K_1 \vee P_{n-1}$ but not of $K_1 \vee \big((t-1)K_2 \cup (n-2t+1)K_1\big)$ and $n$ is sufficiently large relative to the order of $F$. Let $u\in V(G)$ satisfy $d_G(u)=n-1$. Then for every vertex $v\in N_G(u)$,
	\[
	x_v\in\left[\frac{1}{\rho},\frac{1}{\rho}+\frac{2.04}{\rho^2}\right]
	.\]

\end{lemma}
%\vspace*{2mm}
\begin{lemma}\cite{W.H. Wang}.\label{lm10}
	Let $H$ and $H^*$ be the two graphs as shown in Definition \ref{de4}.
	If $n \ge 5 \times 2^{s_2+2} + 1$, then $\rho(K_{1} + H^*) > \rho(K_{1} + H)$.
	
\end{lemma}
%\vspace*{2mm}
%\begin{lemma}\cite{yin2026spectral}.\label{lm11}
%If G is a wheel graph of order $n$, denoted by $W_n$, suppose that $Y = (y_1, y_2, \ldots, y_n)^T$ is the normalized Perron vector of $W_n$, where $y_1$ corresponds to the vertex of degree $n-1.$ then,
 %$$\rho(W_n) = 1 + \sqrt{n} \quad , \quad y_2^2 = \frac{1}{2(n - \sqrt{n})}.$$

%\end{lemma}	
%\vspace*{2mm}	
\noindent \emph{\textbf{Proof of Theorem \ref{thm2}}.}
% \begin{proof}[\textbf{Proof of Theorem \ref{thm2}}]
	If $G'$ is a connected graph with order $n$, by Perron-Frobenius theorem, then there exists a positive eigenvector $X=(x_1,...,x_n)^T$ corresponding to $\rho(G')$.
	Now let $G'$ be an extremal graph to $\textit{spex}_\mathcal{OP}(n,C_{k,l})$. Noting that $C_{k,l}$ is an outerplanar graph not contained in $K_{1,n-1}$, where
	$n \geq \max\{400, \frac{5}{4} |V(F)|\}$. Then $G'$ contains a copy of $K_{1,n-1}$ by Lemma \ref{lm8}.
	For convenience, we now normalize $X$ such that its maximum entry is 1, and then choose $u \in V(G')$ such that $x_{u}=\max\{x_i|i=1,2,...,n\}=1$. Now, we are ready to give a proof of Theorem \ref{thm2}.
%\end{proof}

\begin{claim}\label{claim8}
	$G'[A]$ is a union of disjoint induced paths.
\end{claim}
\noindent \emph{Proof.}	  
Similar to the proof in planar graphs. We first show that $d_{A}(v) \le 2$ for every vertex $v \in A$. If not, then $G'[A \cup \{u\}]$ contains a $K_{3,3}$, a contradiction. Next we assert that there is no cycle in $G'[A]$. If not, suppose there exists a cycle in $G'[A]$. Then we can contract this cycle into a triangle, which implies that $G'$ contains a $K_4$-minor, yielding a contradiction. Thus, $G'[A]$ is a union of disjoint induced paths.
Therefore, $G' \cong K_1+H'$, where $H' \cong G'[A] \cong \bigcup_{i=1}^{t} P_{n_i'}$ with $n_1' \geq n_2' \geq \ldots \geq n_{t}'$ and $t \ge 2$. For convenience, let $P_{n_1'}=v_1v_2 \cdots v_{n_1'}$, $P_{n_2'}=v_{n_1'+1}  \cdots $ $v_{n_1'+n_2'}$ and $P_{n_3'}=v_{n_1'+n_2'+1} \cdots v_{n_1'+n_2'+n_3'}$.

\begin{claim}\label{claim9}
	
$n_2' \leq {l-2}$ and $n_1' \leq {l+k-3}$.
\end{claim}
\noindent \emph{Proof.}	  
%\begin{proof}
We first assume that $n_2' \geq {l-1}$. Then the union of the cycles $uv_1v_2 \cdots v_{l-1}u$ and $uv_{n_1'+1}v_{n_1'+2}$ $ \cdots v_{n_1'+k-1}u$ forms a copy of $C_{k,l}$ in $G$, a contradiction. Hence, $n_2' \leq l-2$. Next, we assume that $n_1' \geq {l+k-2}$. Then the union of the cycles $uv_1v_2 \cdots v_{l-1}u$ and $uv_lv_{l+1} \cdots v_{l+k-2}u$  also forms a copy of $C_{k,l}$ in $G$, a contradiction. Hence, $n_1' \leq {l+k-3}$.
%\end{proof}
\begin{claim}\label{claim10}
$n_i' ={l-2}$ for $i \in \{1,2, \ldots ,t-1\}$.
\end{claim}
\noindent \emph{Proof.}	  
%\begin{proof}
We first suppose that $n_1' \geq {l-1}$. By Claim \ref{claim9}, ${l-1} \leq {n_1'} \leq {l+k-3}$. If $n_2' \geq {k-1}$, then the union of the cycles $uv_1v_2 \cdots v_{l-1}u$ and $uv_{n_1'+1}v_{n_1'+2} \cdots $ $v_{n_1'+k-1}u$ forms a copy of $C_{k,l}$ in $G$, a contradiction. If $n_2' \leq {k-2}$, then let $G'' \cong K_1+H_{\mathcal{OP}}(l-2,l-2)$ is obtained from $G'$ by deleting at most $a'={l+k-4-l+3}={k-1}$ edges and adding at least $b'= \left\lfloor \frac{n-l+2-1}{l-2} \right\rfloor (l-3) - \left\lceil \frac{n-{n_1'}-1}{n_2'} \right\rceil (n_2'-1)$  edges.
It follows that
\[
\begin{aligned}
b'&\geq \left(\frac{n-l+2-1}{l-2}-1\right)(l-3) -  \left(\frac{n-{n_1'}-1}{n_2'}+1\right) (n_2'-1)\\
&= \frac{(l-2-{n_2}')n+(3-2l)(l-3){n_2}'+({n_1}'+1)({n_2}'-1)(l-2)-{n_2}'({n_2}'-1)(l-2)}{(l-2)(k-2)}\\
&\geq \frac{(l-k)n+(3-2l)(l-3)(k-2)-(k-2)(k-3)(l-2)}{(l-2)(k-2)}\quad \text{(since $n_2' \le k-2$)}.
\end{aligned}
\]
By Lemma \ref{lm9} and the Rayleigh principle, 
\[
\begin{aligned}
\rho(G'')-\rho(G') &\geq \frac{‘X^T(A(G'') - A(G'))X}{X^TX}
\geq \frac{{2b'}\frac{1}{\rho^2}-{2a'}\frac{2}{\rho^2}}{X^TX}\\
&\geq \frac{2}{X^TX}\left(\frac{(l-k)n+(3-2l)(l-3)(k-2)-(k-2)(k-3)(l-2)}{(l-2)(k-2)\rho'^2}-\frac{2(k-1)}{\rho'^2} \right) \\
&= \frac{2}{X^TX} \left(\frac{(l-k)n - (k-2)(-2l^2+14l-3kl+6k-19)}{(l-2)(k-2)\rho'^2}
\ \right)\\
&>0,
\end{aligned}
\]
where $n > \max \{ \frac{(k-2)(2l^2 - 14l + 3kl - 6k + 19)}{l - k},26\} $ and $\rho'=\rho(G') \geq \sqrt{n-1}$. So $\rho(G'') > \rho(G')$, contradicting the maximality of $\rho(G')$.
Therefore, $n_1' \leq l-2$. Next, we claim that $n_i'=l-2$ for $i \in \{1,2, \dots , t-1\}$. Suppose there exists an $i_0=\min\{i|1 \leq i \leq t-1,n_i' \leq {l-3}\}$. Let $H'$ be an $(n_{i_0}',n_t')$-transformation. Thus, $n_3'(H') \leq n_2'(H') \leq n_1'(H')=\max\{n_1',n_{i_0}'+1\} \leq {l-2}$, and the longest cycle in $K_1+H'$ is $C_{l-1}$. Therefore, $K_1+H'$ is $C_{k,l}$-free and $\rho(K_1+H')>\rho(G')$ by Lemma \ref{lm10}, a contradiction. Thus, $n_i' ={l-2}$ for $i \in \{1,2, \ldots ,t-1\}$, and $G \cong K_1 + H_{\mathcal{OP}}(l-2,l-2).$
%\end{proof}

Combining the conclusions of Claims \ref{claim8}-\ref{claim10}, we complete the proof of Theorem \ref{thm2}.

\section{Acknowledgement}
We would like to show our great gratitude to anonymous referees for their valuable suggestions which greatly improved the quality of this paper.

\end{document}